\documentclass[conference]{IEEEtran}

\usepackage{amsmath,amssymb,amsthm}
\usepackage{graphicx}
\usepackage{hyperref}
\usepackage{mathtools}
\usepackage[caption=false, font=footnotesize]{subfig}
\usepackage{tikz}

\newtheorem{theorem}{Theorem}[section]
\newtheorem{lemma}[theorem]{Lemma}

\newtheorem{definition}[theorem]{Definition}
\newtheorem{remark}[theorem]{Remark}
\newtheorem{example}[theorem]{Example}

\newcommand{\bbZ}{\mathbb{Z}}

\newcommand{\bbR}{\mathbb{R}}
\newcommand{\bbC}{\mathbb{C}}

\newcommand{\calB}{\mathcal{B}}

\newcommand{\calF}{\mathcal{F}}
\newcommand{\calG}{\mathcal{G}}

\newcommand{\rme}{\mathrm{e}}
\newcommand{\rmi}{\mathrm{i}}

\DeclarePairedDelimiterX\set[2]{\lbrace}{\rbrace}%
 { #1 \,\delimsize| \,\mathopen{} #2 }
\DeclarePairedDelimiterX\abs[1]{\lvert}{\rvert}{#1}
\DeclarePairedDelimiterX\norm[1]{\lVert}{\rVert}{#1}

\DeclareFontFamily{U}{mathx}{\hyphenchar\font45}
\DeclareFontShape{U}{mathx}{m}{n}{
      <5> <6> <7> <8> <9> <10>
      <10.95> <12> <14.4> <17.28> <20.74> <24.88>
      mathx10
      }{}
\DeclareSymbolFont{mathx}{U}{mathx}{m}{n}
\DeclareFontSubstitution{U}{mathx}{m}{n}
\DeclareMathAccent{\widecheck}{0}{mathx}{"71}

\title{Uncovering the limits of uniqueness in sampled Gabor phase retrieval: A dense set of counterexamples in $L^2(\bbR)$}

\author{\IEEEauthorblockN{Rima Alaifari}
\IEEEauthorblockA{ETH Z\"urich \\ Department of Mathematics \\ Email: \href{mailto:rima.alaifari@sam.math.ethz.ch}{rima.alaifari@sam.math.ethz.ch}}
\and
\IEEEauthorblockN{Francesca Bartolucci}
\IEEEauthorblockA{TU Delft \\
Institute of Applied Mathematics \\
Email: \href{mailto:f.bartolucci@tudelft.nl}{f.bartolucci@tudelft.nl}}
\and
\IEEEauthorblockN{Matthias Wellershoff}
\IEEEauthorblockA{University of Maryland \\
Department of Mathematics \\
Email: \href{mailto:wellersm@umd.edu}{wellersm@umd.edu}}}

\begin{document}

\maketitle

\begin{abstract}
    Sampled Gabor phase retrieval --- the problem of recovering a square-integrable signal from the magnitude of its Gabor transform sampled on a lattice --- is a fundamental problem in signal processing, with important applications in areas such as imaging and audio processing. Recently, a classification of square-integrable signals which are not phase retrievable from Gabor measurements on parallel lines has been presented. This classification was used to exhibit a family of counterexamples to uniqueness in sampled Gabor phase retrieval. Here, we show that the set of counterexamples to uniqueness in sampled Gabor phase retrieval is dense in $L^2(\bbR)$, but is not equal to the whole of $L^2(\bbR)$ in general. Overall, our work contributes to a better understanding of the fundamental limits of sampled Gabor phase retrieval.
\end{abstract}

\begin{IEEEkeywords}
    Phase retrieval, Gabor transform, sampling result
\end{IEEEkeywords}

\section{Introduction}

Phase retrieval is a term broadly applied to problems in which information about complex phase needs to be inferred from data. The origins of phase retrieval can be traced back to the years 1915 -- 1929 when W.~H.~Bragg and W.~L.~Bragg (among others) used X-ray diffraction images of crystals in order to illuminate their atomic structure \cite{bragg1915bakerian,bragg1929determination,robertson1937xray,bragg1961rutherford,sayre2002xray}. Since then phase retrieval has found applications in various fields such as crystallography, electron microscopy and astronomy \cite{fienup1982phase,dainty1987phase,shechtman2015phase}.

Gabor phase retrieval refers to problems in which signals $f \in L^2(\bbR)$ have to be reconstructed from magnitudes of their \emph{Gabor transform},
\begin{equation*}
    \mathcal{G} f (x,\omega) := 2^{1/4} \int_\bbR f(t) \rme^{-\pi(t-x)^2} \mathrm{e}^{-2\pi\mathrm{i} t \omega} \, \mathrm{d} t, \quad (x,\omega) \in \bbR^2.
\end{equation*}
It has been used in a range of audio processing tasks, including the phase vocoder for time-stretching and pitch-shifting of audio signals \cite{prusa2017phase}, as well as speech enhancement and source separation \cite{abdelmalek2022audio}.

In this contribution, we will specifically focus on the \emph{sampled Gabor phase retrieval} problem which is the recovery of signals $f$ from the magnitude measurements $(\abs{\calG f (x,\omega)})_{(x,\omega) \in \Lambda}$ where $\Lambda \subset \bbR^2$ is a lattice\footnote{A \emph{lattice} $\Lambda \subset \bbR^2$ is a discrete subset of the time-frequency plane that can be written as $L \bbZ^{k}$ where $L \in \bbR^{2 \times k}$ is a matrix with linearly independent columns and $k \in \{1,2\}$.}. We focus on this sampling problem because magnitude information on the entire time-frequency plane $\bbR^2$ is not available in practice. Instead, only a finite number of measurements are stored and inferences are made based on them. We consider the sampled setup proposed above as a natural and useful compromise between the fully continuous case, where no sampling occurs, and the fully discrete case, where the signals are finite-dimensional vectors.

\subsection{Prior arts: Counterexamples to uniqueness in sampled Gabor phase retrieval}

In the following, we will focus on counterexamples to uniqueness in sampled Gabor phase retrieval; i.e.~signals whose Gabor transform magnitudes agree on a lattice but which are fundamentally different from each other. Before introducing the concept of a counterexample rigorously, we need to emphasise that there is one ever-present ambiguity in Gabor phase retrieval: the global phase ambiguity. Two signals $f,g \in L^2(\bbR)$ are said to \emph{agree up to global phase} if they are equivalent with respect to the relation 
\begin{equation*}
    f \sim g :\iff f = \rme^{\rmi \alpha} g, \mbox{ for some } \alpha \in \bbR.
\end{equation*}
Notably, $f \sim g$ implies $\abs{\calG f} = \abs{\calG g}$ such that signals which agree up to global phase cannot be recovered from Gabor transform magnitudes. With this in mind, we define the set of counterexamples.

\begin{definition}[Counterexamples]\label{def:counterexamples}
    Let $\Lambda \subset \bbR^2$. The \emph{set of counterexamples} to uniqueness in Gabor phase retrieval on $\Lambda$ is defined by
    \begin{multline*}
        \mathfrak{C}(\Lambda) := \left\{ f \in L^2(\bbR) \,\middle| \, \abs{\calG f} = \abs{\calG g} \mbox{ on } \Lambda \mbox{ and } f \not\sim g, \right. \\
        \left. \mbox{for some } g \in L^2(\bbR) \right\}.
    \end{multline*}
    An element $f \in \mathfrak{C}(\Lambda)$ is called a \emph{counterexample} to uniqueness in sampled Gabor phase retrieval on $\Lambda$.
\end{definition} 

Counterexamples to uniqueness in sampled Gabor phase retrieval are interesting for two reasons. First, they allow us to better understand the fundamental limits of sampled Gabor phase retrieval, which in turn can guide future research towards achieving uniqueness. In addition, they offer the opportunity to explore the potential relationship between uniqueness and stability in phase retrieval \cite{alaifari2022connection}.

Let us briefly summarise the recent research on counterexamples in Gabor phase retrieval. The relationship between the Gabor transform and the Bargmann transform (which is described in more detail in Section~\ref{sec:Gabor_Bargmann}) allows for the relation of the Gabor phase retrieval problem to a phase retrieval problem for entire functions. This was realised in \cite{mcdonald2004phase,jaming2014uniqueness}. Then, following these ideas, a characterisation of all entire functions of exponential-type whose magnitudes agree on any set of infinitely many equidistant parallel lines was proven in \cite{wellershoff2023phase}.

Using this characterisation and noting that all lattices are a subset of some set of infinitely many equidistant parallel lines, it becomes possible to construct various types of counterexamples. This idea has been applied in \cite{alaifari2021phase} to construct explicit counterexamples to uniqueness in sampled Gabor phase retrieval on any lattice. (See \cite{wellershoff2023phase} for a more in-depth explanation.) An extension of the results in \cite{alaifari2021phase} has appeared in \cite{grohs2022foundational}.

\subsection{Our contributions}

In this contribution, we show that the set of counterexamples $\mathfrak{C}(\Lambda)$ is dense in $L^2(\bbR)$ when $\Lambda \subset \bbR^2$ is a lattice or a set of equidistant parallel lines. We also show that the Gaussian is \emph{not} a counterexample for quadratic lattices, $\Lambda = a \bbZ^2$, with $a \in (0,1)$. Therefore, \emph{the set of counterexamples is dense but not equal to the whole of $L^2(\bbR)$ in general}. We prove these two results by using the connection between the Bargmann transform and the Gabor transform as well as some classical results from complex analysis.

Note that this contribution is a condensed and modified version of the section on the fragility of uniqueness in sampled Gabor phase retrieval in the larger manuscript \cite{alaifari2022connection}. Apart from a comprehensive treatment of counterexamples, the larger manuscript also discusses the stability of Gabor phase retrieval as well as its potential connection with uniqueness in sampled Gabor phase retrieval. Here, we focus specifically on showing that the counterexamples are dense.

\subsection*{Notation}

Rotation by $\theta \in \bbR$ on $\bbR^2$ is denoted by $\operatorname{R}_\theta : \bbR^2 \to \bbR^2$; in matrix notation, we have
\begin{equation*}
    R_\theta = \begin{pmatrix}
        \cos \theta & - \sin \theta \\
        \sin \theta & \cos \theta
    \end{pmatrix}.
\end{equation*}
Translation by $x \in \bbR$ on $L^p(\bbR)$, where $p \in [1,\infty]$, is denoted by $\operatorname{T}_x : L^p(\bbR) \to L^p(\bbR)$; i.e.
\begin{equation*}
    \operatorname{T}_x f(t) = f(t-x), \qquad t \in \bbR,
\end{equation*}
for $f \in L^p(\bbR)$. Finally, the normalised Gaussian is denoted by
\begin{equation*}
    \varphi(t) = 2^{1/4} \mathrm{e}^{-\pi t^2}, \qquad t \in \bbR.
\end{equation*}

\section{The relation between the Bargmann and Gabor transform}\label{sec:Gabor_Bargmann}

As mentioned before, we will make use of the well-known connection between the Bargmann transform and the Gabor transform \cite{grochenig2000foundations}. The \emph{Fock space} $\mathcal{F}^2(\bbC)$ is the Hilbert space of all entire functions for which the norm induced by the inner product
\begin{equation*}
    (F,G)_\mathcal{F} := \int_\bbC F(z) \overline{G(z)} \mathrm{e}^{-\pi \lvert z \rvert^2} \,\mathrm{d} z
\end{equation*}
is finite. The \emph{Bargmann transform} $\mathcal{B} : L^2(\bbR) \to \mathcal{F}^2(\bbC)$,
\begin{equation*}
    \mathcal{B} f (z) := 2^{1/4} \int_\bbR f(t) \mathrm{e}^{2 \pi t z - \pi t^2 - \frac{\pi}{2} z^2} \,\mathrm{d} t, \quad z \in \bbC,
\end{equation*}
acts as an isomorphism between $L^2(\bbR)$ and $\mathcal{F}^2(\bbC)$. It is related to the Gabor transform by the formula
\begin{equation}\label{eq:Gabor_Bargmann}
    \mathcal{G} f(x,-\omega) = \mathrm{e}^{\pi\mathrm{i}x\omega} \mathcal{B} f(x + \mathrm{i} \omega) \mathrm{e}^{-\frac{\pi}{2}(x^2 + \omega^2)},
\end{equation}
for $(x,\omega) \in \bbR^2$. It is the formula above that allows us to relate the Gabor phase retrieval problem to a phase retrieval problem for entire functions. We are going to use this relation in the sequel.

\section{The set of counterexamples is dense}\label{sec:density}

In this section, we show that the set of counterexamples $\mathfrak{C}(\Lambda)$ is dense in $L^2(\bbR)$ when $\Lambda \subset \bbR^2$ is a lattice or a set of equidistant parallel lines. Our strategy will be to design entire functions $H^\pm_\delta \in \calF^2(\bbC)$ which converge to a constant as $\delta \to 0$, do not agree up to global phase and still satisfy $\abs{H^+_\delta} = \abs{H^-_\delta}$ on $\bbR + \mathrm{i} a \mathbb{Z}$. Once these entire functions have been designed, we can multiply them with $\calB f$ and take the inverse Bargmann transform in order to get counterexamples on $\bbR \times a \bbZ$ that are close to $f \in L^2(\bbR)$.

We design the functions $H^\pm_\delta \in \calF^2(\bbC)$ by modifiying the counterexamples on $\bbR \times a \bbZ$,
\begin{equation*}
    h^\pm(t) := \varphi(t) \left( \cosh\left( \frac{\pi t}{a} \right) \pm \mathrm{i} \sinh\left( \frac{\pi t}{a} \right) \right), \quad t \in \bbR,
\end{equation*}
presented in \cite{alaifari2021phase}. To accomplish this, we compute the Bargmann transforms of $h^\pm$ which are given by
\begin{equation*}
    z \mapsto \left( 1 \mp \rmi + (1 \pm \rmi) \rme^{\frac{\pi z}{a}} \right) \rme^{-\frac{\pi z}{2a}}
\end{equation*}
up to a constant depending on $a$. Next, we note that time-shifting $h^\pm$ by $u \in \bbR$ will produce additional counterexamples on $\bbR \times a \bbZ$ according to the covariance property of the Gabor transform \cite{grochenig2000foundations}. For $f \in L^2(\bbR)$, it holds that 
\begin{equation*}
    \calB \operatorname{T}_u f (z) = \calB f (z - u) \rme^{\pi u z} \rme^{-\frac{\pi}{2} u^2}, \quad z \in \bbC,
\end{equation*}
such that the Bargmann transforms of the counterexamples $\operatorname{T}_u h^\pm$, with $u = - \tfrac{a}{\pi} \log \delta$, are given by 
\begin{equation*}
    z \mapsto \left( 1 \mp \rmi + (1 \pm \rmi) \delta \cdot \rme^{\frac{\pi z}{a}} \right) \delta^{-a z} \rme^{-\frac{\pi z}{2a}}
\end{equation*}
up to a constant depending on $a$ and $\delta$. After multiplying by $\delta^{a z} \rme^{\frac{\pi z}{2a}}$, it follows that 
\begin{equation*}
    z \mapsto 1 \mp \rmi + (1 \pm \rmi) \delta \cdot \rme^{\frac{\pi z}{a}}
\end{equation*}
are entire functions whose magnitudes agree on $\bbR + \rmi a \bbZ$ (which do not agree up to global phase) and we define 
\begin{equation}\label{eq:Fock_multipliers}
    H^\pm_\delta(z) := 1 \pm \mathrm{i} \delta \cdot \mathrm{e}^{\tfrac{\pi z}{a}}, \quad z \in \bbC,
\end{equation}
after multiplying by $(1\pm\rmi)/2$.

\begin{theorem}\label{thm:main_proof}
    Let $a > 0$. Then, $\mathfrak{C}(\bbR \times a \bbZ)$ is dense in $L^2(\bbR)$.
\end{theorem}

\begin{proof}
    Let $\epsilon > 0$ and $f \in L^2(\bbR)$. We want to show that there exist $g^\pm \in L^2(\bbR)$ which do not agree up to global phase, are $\epsilon$-close to $f$ in $L^2(\bbR)$ and satisfy 
    \begin{equation*}
        \lvert \mathcal{G} g^+ \rvert = \lvert \mathcal{G} g^- \rvert \mbox{ on } \bbR \times a \bbZ.
    \end{equation*}
    To do so, we note that the monomials
    \begin{equation*}
        e_n(z) := \left(\frac{\pi^n}{n!}\right)^{1/2} z^n, \qquad n \in \mathbb{N}_0,~z \in \bbC,
    \end{equation*}
    form an orthonormal basis for the Fock space $\mathcal{F}^2(\bbC)$ \cite{grochenig2000foundations}. Therefore, the space of complex polynomials is dense in $\mathcal{F}^2(\bbC)$ and we can find $P \in \bbC[z]$ such that 
    \begin{equation*}
        \norm*{ \mathcal{B} f - P }_{\mathcal{F}} < \frac{\epsilon}{2}.
    \end{equation*}
    Let us now consider the entire functions $H^\pm_\delta$ defined in equation~\eqref{eq:Fock_multipliers} and note that $G_\delta^\pm := H_\delta^\pm \cdot P \in \mathcal{F}^2(\bbC)$ since $G^\delta_\pm$ are entire functions of exponential-type. Hence, we can define the signals $g^\pm_\delta := \mathcal{B}^{-1} G^\pm_\delta \in L^2(\bbR)$. To establish the desired properties of $g^\pm_\delta$, we will work with their Bargmann transforms $G_\delta^\pm$. First, we note that $\abs{H^+_\delta} = \abs{H^-_\delta}$ on $\bbR + \mathrm{i} a \mathbb{Z}$ implies $\lvert G^+_\delta \rvert = \lvert G^-_\delta \rvert$ on $\bbR + \mathrm{i} a\bbZ$ and thus $\lvert \mathcal{G} g^+_\delta \rvert = \lvert \mathcal{G} g^-_\delta \rvert$ on $\bbR \times a \bbZ$ by equation~\eqref{eq:Gabor_Bargmann}. Secondly, we note that the entire functions $G^\pm_\delta$ do not agree up to global phase: indeed, both entire functions $H^\pm_\delta$ have infinitely many roots but no root of $H^+_\delta$ is a root of $H^-_\delta$ and vice versa. At the same time, $P$ is a polynomial and has only finitely many roots. It follows that $G^+_\delta$ does have roots which are not roots of $G^-_\delta$ (and vice versa) and thus $G^+_\delta \not \sim G^-_\delta$. By the linearity of the Bargmann transform, we can conclude that $g^+_\delta \not\sim g^-_\delta$. Finally, we note that the definition of $H_\delta^\pm$ in equation~\eqref{eq:Fock_multipliers} directly implies that
    \begin{equation*}
        \norm{ P - P \cdot H^\pm_\delta }_\mathcal{F} = \delta \norm{ z \mapsto P(z) \cdot \mathrm{e}^{\pi z/a} }_\mathcal{F},
    \end{equation*}
    and so there exists a $\delta > 0$ depending on $a$, $\epsilon$ and $P$ (which in turn depends on $f$ and $\epsilon$) such that 
    \begin{equation*}
        \norm{ P - P \cdot H^\pm_\delta }_\mathcal{F} < \frac{\epsilon}{2}.
    \end{equation*}
    We conclude 
    \begin{align*}
        \norm{ f - g^\pm_\delta }_2 &= \norm{ \mathcal{B} f - H^\delta_\pm \cdot P }_\mathcal{F} \\
        &\leq \norm{ \mathcal{B} f - P }_\mathcal{F} + \norm{ P - H^\delta_\pm \cdot P }_\mathcal{F}
        < \epsilon.
    \end{align*}
\end{proof}

\begin{remark}[Some explanations on the proof]
    \label{rem:interesting_examples}
    As $\mathcal{B} f \in \mathcal{F}^2(\bbC)$, for $f \in L^2(\bbR)$, we know that $\mathcal{B} f$ is either an entire function of exponential-type or an entire function of second order. If $\mathcal{B} f$ is of second order, then its type is either strictly smaller than $\pi/2$ or exactly $\pi/2$. In most of these cases, it holds that $\mathcal{B} f \cdot H^\pm_\delta \in \mathcal{F}^2(\bbC)$ and thus we can define 
    \begin{equation*}
        g^\pm_\delta := \mathcal{B}^{-1} \left( \mathcal{B} f \cdot H^\pm_\delta \right) \in L^2(\bbR),
    \end{equation*}
    with 
    \begin{equation*}
        \delta < \frac{\epsilon}{\norm{z \mapsto \mathcal{B} f (z) \mathrm{e}^{\pi z/a} }_\mathcal{F}},
    \end{equation*}
    to obtain counterexamples which are $\epsilon$-close to $f$ in $L^2(\bbR)$. We note that $g^\pm_\delta$ are small additive perturbations of our original signals $f$.

    Unfortunately, there is one case in which this simple strategy does not work: the one in which $\mathcal{B} f$ is a second-order entire function of type $\pi/2$. Indeed, in this case, it is not guaranteed that $\mathcal{B} f \cdot H^\pm_\delta$ is in the Fock space. --- Two striking examples for why this can fail can be found in \cite{beneteau2010extremal}. --- As the only situation in which $\mathcal{B} f \cdot H^\pm_\delta$ is not in the Fock space occurs when $\mathcal{B} f$ is exactly of order two and of type $\pi/2$, it seems obvious that the functions $f$ for which $\mathcal{B} f \cdot H^\pm_\delta \in \mathcal{F}^2(\bbC)$ holds must be dense in $L^2(\bbR)$. We can prove this by realising that the complex polynomials are dense in $\mathcal{F}^2(\bbC)$. 
\end{remark}

Theorem~\ref{thm:main_proof} continues to hold for any set of infinitely many equidistant parallel lines. We can show this by considering the entire functions
\begin{equation*}
    H^\pm_\delta(z) := 1 \pm \mathrm{i} \delta \exp\left( \frac{\pi \mathrm{e}^{\mathrm{i} \theta}}{a} \left( z - \overline \lambda_0 \right) \right)
\end{equation*}
and realising that the corresponding signals $g^\pm_\delta \in L^2(\bbR)$ satisfy 
\begin{equation*}
    \lvert \mathcal{G} g^+_\delta \rvert = \lvert \mathcal{G} g^-_\delta \rvert \mbox{ on } \operatorname{R}_\theta \left( \bbR \times a \bbZ \right) + \lambda_0,
\end{equation*}
where $a > 0$, $\lambda_0 \in \bbR^2 \simeq \bbC$. The statement for general lattices follows from the same consideration because all lattices are subsets of some set of infinitely many equidistant parallel lines. We therefore arrive at the following result.

\begin{theorem}\label{thm:main}
    Let $\Lambda \subset \bbR^2$ be a set of equidistant parallel lines or a lattice. Then, $\mathfrak{C}(\Lambda)$ is dense in $L^2(\bbR)$.
\end{theorem}

To illustrate our main results, we construct counterexamples that are close to the Hermite functions and plot their spectrograms.

\begin{example}
    Consider the $n$-th Hermite function $H_n \in L^2(\bbR)$ given by
    \begin{equation*}
        \mathcal{B} H_n(z) = e_n(z) = \left(\frac{\pi^n}{n!}\right)^{1/2} z^n, \qquad z \in \bbC.
    \end{equation*}
    By equation~\eqref{eq:Gabor_Bargmann}, the Gabor transform of the Hermite function is
    \begin{align*}
        \mathcal{G} H_n(x,\omega) &= \mathrm{e}^{-\pi\mathrm{i} x \omega} \mathcal{B} H_n(x - \mathrm{i} \omega) \mathrm{e}^{-\frac{\pi}{2}\left( x^2 + \omega^2 \right)} \\
        &= \left(\frac{\pi^n}{n!}\right)^{1/2} \mathrm{e}^{-\pi\mathrm{i} x \omega} \left( x - \mathrm{i} \omega \right)^n \mathrm{e}^{-\frac{\pi}{2}\left( x^2 + \omega^2 \right)},
    \end{align*}
    for $(x,\omega) \in \bbR^2$. If we plot the magnitude of the above (for $n = 5$), we obtain Figure~\ref{fig:Gabor_Hermite}. Next, we want to find counterexamples which are close to $H_n$. According to Remark~\ref{rem:interesting_examples}, we can define $g^+_\delta := \mathcal{B}^{-1} ( \mathcal{B} H_n \cdot H^+_\delta )$. Let us visualise the spectrogram of $g^+_\delta$, i.e.
    \begin{equation*}
        \mathcal{G} g^+_\delta(x,\omega) = \mathcal{G} H_n(x,\omega) \cdot H^+_\delta(x - \mathrm{i} \omega), \quad (x,\omega) \in \bbR^2,
    \end{equation*}
    in Figure \ref{fig:counterexample} (for $n = 5$, $a = \tfrac{1}{4}$ and $\delta = \tfrac{1}{50} \exp(-10 \pi)$).
\end{example}

\begin{figure}
    \centering
    \subfloat[$\lvert \mathcal{G} H_5 \rvert$]{\label{fig:Gabor_Hermite}
        \includegraphics[width=.4\textwidth]{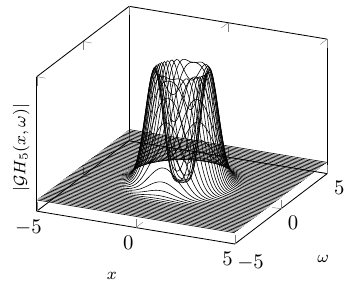}}
    \\
    \subfloat[$\lvert \mathcal{G} g_\delta^+ \rvert$, for $\delta = \tfrac{1}{50} \exp(-10 \pi)$.]{\label{fig:counterexample}
        \includegraphics[width=.4\textwidth]{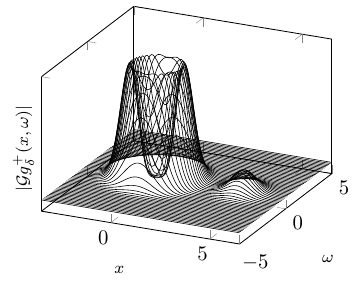}}
    \caption{The Gabor magnitude of the fifth Hermite function (Figure~\ref{fig:Gabor_Hermite}) and of a counterexample $g_\delta^+$ to sampled Gabor phase retrieval on $\bbR \times \tfrac{1}{4} \bbZ$ (Figure~\ref{fig:counterexample}).}
    \label{fig:somename}
\end{figure}

\section{The Gaussian is not a counterexample}\label{sec:Gauss}

Finally, we can show that the Gaussian is not a counterexample on $a \bbZ^2$ if $a \in (0,1)$. Specifically, we prove the following result.

\begin{theorem}\label{thm:original}
    Let $0 < a < 1$ and $f \in L^2(\bbR)$ be such that 
    \[
        \lvert \mathcal{G} f(x,\omega) \rvert^2  = \lvert \mathcal{G} \varphi(x,\omega) \rvert^2, \qquad (x,\omega) \in a \mathbb{Z}^2.
    \]
    Then, there exists an $\alpha \in \bbR$ such that $f = \mathrm{e}^{\mathrm{i} \alpha} \varphi$.
\end{theorem}

Since the Bargmann transform of the Gaussian is one, equation~\eqref{eq:Gabor_Bargmann} implies that the theorem above is equivalent to the following lemma.

\begin{lemma}
    \label{lem:restated}
    Let $0 < a < 1$ and let $F \in \mathcal{F}^2(\mathbb{C})$ be such that 
    \[
        \lvert F(z)  \rvert = 1 = \lvert \mathcal{B} \varphi(z) \rvert, \qquad z \in a \mathbb{Z} + \mathrm{i} a \mathbb{Z}.
    \]
    Then, there exists an $\alpha \in \bbR$ such that $F = \mathrm{e}^{\mathrm{i} \alpha}$.
\end{lemma}

The intuition for the proof of this lemma comes from the maximum modulus principle: we note that we are considering a second order entire function $F$ which is bounded on all lattice points; this suggests that $F$ should be constant in the entire complex plane as long as the lattice is dense enough. This intuition is indeed correct, as evidenced by the following result independently discovered by V.~G.~Iyer \cite{iyer1936note} and A.~Pfluger \cite{pfluger1937analytic}.

\begin{theorem}
    \label{thm:coreresult}
    Let $h$ be an entire function such that 
    \[
        \limsup_{r \to \infty} \frac{\log M_h(r)}{r^2} < \frac{\pi}{2},
    \]
    where $M_h(r) := \max_{\lvert z \rvert = r} \lvert h(z) \rvert$. If there exists 
    a constant $\kappa > 0$ such that 
    \[
        \lvert h(m+\mathrm{i} n) \rvert \leq \kappa, \qquad m,n \in \mathbb{Z},
    \]
    then $h$ is constant.
\end{theorem}

\begin{proof}[Proof of Lemma~\ref{lem:restated}]
    Consider the function $h(z) := F(az)$, for $z \in \mathbb{C}$. It holds that 
    \[
        \lvert h(z) \rvert = \lvert F(az) \rvert \leq \norm{ F }_\mathcal{F} \cdot  \mathrm{e}^{\frac{\pi}{2} \lvert az \rvert^2} 
        = \norm{ F }_\mathcal{F} \cdot \mathrm{e}^{\frac{\pi a^2}{2} \lvert z \rvert^2},
    \]
    for $z \in \bbC$, such that 
    \begin{align*}
        \limsup_{r \to \infty} \frac{\log M_h(r)}{r^2} &\leq \limsup_{r \to \infty} \left(\frac{\log \lVert F \rVert_\mathcal{F}}{r^2} + \frac{\pi a^2}{2} \right) \\
        &= \frac{\pi a^2}{2} < \frac{\pi}{2}.
    \end{align*}
    Additionally,
    \[
        \lvert h(m+\mathrm{i} n) \rvert = \lvert F(am + \mathrm{i} a n) \rvert = 1, \qquad m,n \in \mathbb{Z},
    \]
    holds such that the assumptions of Theorem \ref{thm:coreresult} are met and we can 
    conclude that $h$ is constant. As $\lvert h(0) \rvert = 1$, it follows that there must exist an $\alpha 
    \in \bbR$ such that $h = \mathrm{e}^{\mathrm{i} \alpha}$ which implies $F = \mathrm{e}^{\mathrm{i} \alpha}$.
\end{proof}

We have therefore shown that the set of counterexamples is not equal to the whole of $L^2(\bbR)$ when $\Lambda$ is a sufficiently dense quadratic lattice.

\begin{remark}
    A natural confusion that might arise in connection with Theorem~\ref{thm:original} is in how far it is different from the result in \cite{grohs2023injectivity} on shift-invariant spaces with Gaussian generator $V_\beta^1(\varphi)$. While Theorem~\ref{thm:original} implies that the Gaussian can be distinguished from all other functions in $L^2(\bbR)$ based on its sampled Gabor magnitude measurements, the result in \cite{grohs2023injectivity} only implies that it can be distinguished from the functions in $V_\beta^1(\varphi) \subset L^2(\mathbb{R})$.
\end{remark}

\section*{Acknowledgements}

The authors would like to extend their heartfelt thanks to Stefan Steinerberger for his insightful discussions and acknowledge funding through the SNSF grant 200021\_184698.

\bibliographystyle{IEEEtran}
\bibliography{IEEEabrv,bibfile_final}

\end{document}